\documentclass[a4paper, 12pt]{article}

\usepackage{cite}
\usepackage{amsmath}
\usepackage{amsthm}
\usepackage[all]{xy}
\usepackage{amssymb}
\usepackage{mathrsfs}

\theoremstyle{plain}
\newtheorem{thm}{Theorem}[section]
\newtheorem{prop}[thm]{Proposition}
\newtheorem{lem}[thm]{Lemma}
\newtheorem{cor}[thm]{Corollary}
\newtheorem{fct}[thm]{Fact}

\newtheorem{claim}[thm]{Claim}
\newtheorem{question}[thm]{Question}

\theoremstyle{definition}
\newtheorem{defn}[thm]{Definition}
\newtheorem{exmp}[thm]{Example}

\newtheorem{rmk}[thm]{Remark}

\newcommand{\ovl}{\overline}

\newcommand{\sset}{\subset}
\newcommand{\ssetq}{\subseteq}

\newcommand{\restr}{\upharpoonright}

\newcommand{\tit}{\textit}

\newcommand{\mV}{\mathcal{V}}
\newcommand{\mW}{\mathcal{W}}

\newcommand{\mVW}{\mathcal{V}\times\mathcal{W}}
\newcommand{\WAV}{\mW\times_{\mathbb{A}^1}\mV}
\newcommand{\WSV}{\mW\times_{S}\mV}

\renewcommand{\phi}{\varphi}
\newcommand{\gen}[1]{\left\langle#1\right\rangle}
\newcommand{\indep}[1][]{\mathop{\raisebox{-.9ex}{$\underset{#1}{\smile}$}\makebox[-2.3ex]{$\mid$}\hspace{2.3ex}}} %independence, optional parameter: over ...
\usepackage{color}

\title{A Note on Integrability and Internality in DCF$_0$}
\author{J. Nagloo$^1$, D. Penazzi$^2$\\  University of Leeds}
\date{}
\def\subjclass#1{{\renewcommand{\thefootnote}{}%
\footnote{\emph{2010 Mathematics Subject Classification.} #1}}}

\begin{document}

\maketitle 

\subjclass{Primary 34M15, 14H70, 03C98}

\begin{abstract}
We investigate the relationship between algebraic integrability and the model theoretic notion of internality. Our main result give a geometric account of almost internality and indeed we show that this notion correspond in a reasonable way to having enough ``good'' first intergrals.
\end{abstract}
\footnotetext[1]{Supported by an EPSRC Project Studentship and a University of Leeds - School of Mathematics partial scholarship}
\footnotetext[2] {Supported by a postdoctoral fellowship on EPSRC grant EP/I002294/1}

\section{Introduction}

The notion of integrability is used in many area of mathematics but it is not unique, there is in fact a cloud of possible definitions, often based on the context we start with (e.g. if the equation is in Hamiltonian form, or if there exists a Poisson bracket). Our context will be that of algebraic differential equations, and the notion of integrability will be that of having ``enough'' first integrals.

In the model theory of differential fields we have suitable candidates for a definition of integrability in the notions of internality and almost internality of the solution set to the constants. Very little has been done in the direction of understanding such notions from the point of view of other mathematical disciplines, some results are in the work of  Chatzidakis-Hrushovski \cite{Chat} and Pillay \cite{Pillay1}, but no work in the literature is completely devoted to this issue. This article aims to fill this gap and give a geometric description of what does it mean for the solutions of a differential equation to be (almost) internal to the constants in a language understandable by the general mathematician.

We begin by recalling briefly some model theory; we direct the reader to \cite{Marker} for more details. We fix the language $L_{\partial}$=($+,-,\cdot,0,1,\partial$) of differential fields. $DCF_0$, the theory of differentially closed fields of characteristic 0 is given by the axioms of differential fields together with the axioms saying that for a differential field $K$, any finite system of differential polynomial equations over $K$ with a solution in some extension of $K$ already has a solution in $K$. $DCF_0$ is complete, has quantifier elimination, is $\omega$-stable and has elimination of imaginaries. A differentially closed field is precisely a model of $DCF_0$. For $K$ a differentially closed field, the basic objects of interest are affine differential algebraic varieties (or Kolchin closed sets): common solution sets in $K^n$ of finite systems of differential polynomial equations over $K$. Quantifier elimination for $DCF_0$ implies that, up to finite boolean combinations, the Kolchin closed sets are precisely the definable sets in $K$. If a definable set $Y$ of $K^n$ is defined by $L_{\partial}$-formulas with parameters from a differential subfield $k$ of $K$ we will say $Y$ is defined over $k$.\\

We fix a universal domain (in model-theoretic terms, a saturated model) $(\mathcal{U},\partial)$ of $DCF_0$ containing the field of complex number; this is possible by $\omega$-stability, we refer to Theorem 4.3.15 of \cite{Marker} for details. We denote its field of constants by $\mathcal{C}$. Throughout the article $t$ will denote an element of $\mathcal{U}$ with the property that $\partial t=1$. If $Y\subseteq\mathcal{U}^n$ is definable (over a differential subfield $k$ say), then $Y$ is said to be \emph{finite-dimensional} if the order of $Y$, $ord(Y)=sup\{tr.deg(k\gen{y}/k):y\in Y\}$, is finite. Here $k\gen{y}$ denote the differential subfield of $\mathcal{U}$ generated by $k$ and $y$, i.e. $k(y,\partial y,\partial^2 y,\ldots)$.\\

Given $A\subset\mathcal{U}^n$, we say that $a\in \mathcal{U}$ is (model-theoretically) \emph{algebraic over} $A$ if there is a $L_{\partial}$-formula $\varphi(x)$ with parameter from $A$ such that $\varphi(a)$ is true in $\mathcal{U}$ and $\varphi(\mathcal{U})$ is finite. In the case where $\varphi(\mathcal{U})=\{a\}$, we say that $a$ is \emph{definable over} $A$. We denote by $acl(A)$ (resp. $dcl(A)$) the set of all elements of $\mathcal{U}$ which are algebraic (resp. definable) over $A$. It is a nontrivial fact that $dcl(A)$ is the differential subfield of $\mathcal{U}$ generated by $A$ and that $acl(A)$ is $dcl(A)^{alg}$, the field-theoretic algebraic closure of $dcl(A)$.\\ 

Let $X$ and $Y$ be $\emptyset$-definable sets in $\mathcal{U}$. Then:
\begin{enumerate}
\item $X$ is {\em internal} to $Y$ if there is a definable surjective map $f$ from $Y^m$ to $X$ (for some $m$).
\item $X$ is {\em almost internal} to $Y$ if there is a definable relation $R\subset Y^m\times X$ such that for any $x\in X$, there are only finitely many, but at least one, $y\in Y^m$ such that $R(y,x)$.
\end{enumerate}

Note that the definable function/relation might be definable with some extra parameters. So sometimes, we also say that $X$ is internal (resp. almost internal) to $Y$ over $A\subset\mathcal{U}$ if the definable function (resp. definable relation) is defined over $A$.

Internality plays an important role in the model theoretic treatment of differential Galois theory. Indeed for any $\omega$-stable theory and for any $\emptyset$-definable sets $X$ and $Y$ such that $X$ is {\em internal} to $Y$, if we let $Gal(X/Y)$ be the group of permutation of $X$ induced by automorphisms of a universal model of the theory which fix $Y$ pointwise, then there is a $\emptyset$-definable group $G$ and a $\emptyset$-definable group action of $G$ on $X$ which is isomorphic to the action of $Gal(X/Y)$ on $X$. In $DCF_0$ this gives a very general treatment of the Picard-Vessiot theory and the Kolchin theory (and more): analogues of Galois theory for differential fields. We direct the reader to the survey \cite{Pillay4} and references therein, for more details.
In this paper, though, we do not deal with Galois theory.\\
 As we shall see soon, we will work birationally, using the weaker notion of "generic" internality/almost internality. 
\vskip 0.3 cm

The main result (Theorem \ref{mainth}) can be restated roughly as follows (see Theorem \ref{restthm}): a differential algebraic variety of order $n$, viewed as a $\mathbb{C}(t)$-definable set in $\mathcal{U}$ ($\partial t=1$), is generically almost internal to the constants if and only if the corresponding differential equation in $y(t)$: $F(t,y,y',\dots,y^{(n)})=0$, has $n$ independent first integrals that are rational functions in $t,y,y',\dots,$ $y^{(n-1)},a_1,\dots,a_k$, where $a_1,\dots,a_k$ are some particular solutions of the differential equation.
\vskip 0.3 cm
In the next two sections (2 and 3) we will introduce all the needed notions. Section $4$ contains the statement and proof of our main result, while in the last section we give some further results and remarks.

\section{Algebraic $D$-varieties.}
Although Kolchin's differential algebraic geometry, the study of (not necessarily affine) differential algebraic varieties, is very similar to the model theory of differential field, we need to work in a more geometric context to find a connection with integrability. 
Thus we use the alternative but almost equivalent notion of a $D$-variety introduced by Buium in \cite{Buium}. The category of $D$-varieties is in fact birationally equivalent to the category of finite dimensional differential algebraic varieties.
%{\color{red} It turns out that for finite dimensional sets, Buium in \cite{Buium}, introduced the alternative but almost equivalent notion of algebraic $D$-varieties which we now recall.}\\

Let $K$ be a differential field of characteristic zero and let $V\subseteq \mathcal{U}^n$ be a (not necessarily irreducible) affine algebraic variety defined over $K$. The {\em shifted tangent bundle} is defined to be
\begin{equation*}
T_{\partial}(V)=\left\{(a,u)\in \mathcal{U}^{2n}\; :\; a\in V, \sum_{i=1}^n \frac{\partial P}{\partial x_i}(a)u_i+ P^{\partial}(a)=0 \;\;for\; p\in I(V)\right\}
\end{equation*}
where $I(V)\subset K\{x_1,\ldots,x_n\}$ is the ideal of $V$ and $ P^{\partial}$ is the polynomial obtained by differentiating the coefficients of $P$. For an arbitrary abstract variety $V$, we take a covering of $V$ by affine opens $U_i$ and piece together the shifted tangent bundles $T_{\partial}(U_i)$ using the obvious transition maps to obtain $T_{\partial}(V)$. A {\em shifted vector field} on $V$ over $K$ is just a morphism $s: V\rightarrow T_{\partial}(V)$ which is also a section of the canonical projection $\pi: T_{\partial}(V)\rightarrow V$. Clearly, if $V$ is defined over $\mathcal{C}$, then  $T_{\partial}(V)$ is just the tangent bundle, $T(V)$, of $V$. In this case $s$ is a ``usual" vector field, i.e. a morphism  $V \rightarrow T(V)$.

A pair $(V,s)$, where $s$ is a polynomial, is called an {\em algebraic $D$-variety} over $K$ and if $s$ is a rational section, we refer to the pair as a {\em rational $D$-variety}.

\begin{rmk}
Buium's original definition is actually the following: let $V$ be an algebraic variety, defined over a differential field $K$, a structure on $V$ of an algebraic $D$-variety over $K$ is given by an extension $D_V$ of $\partial$ to a derivation of the structure sheaf $\mathcal{O}_V$. The the pair $(V,D_V)$ is then an algebraic $D$-variety over $K$.\\
It is not hard to see that $D_V$ corresponds to a regular section $s$ of the shifted tangent bundle of $V$.
\end{rmk}
Now, given an algebraic (or rational) $D$-variety $(V,s)$ over a differential field $K$, we define $(V,s)^{\partial}$ to be $\left\{a\in V(\mathcal{U})\; : \partial(a)=s(a)\right\}$. $(V,s)^{\partial}$ is then a finite dimensional differential algebraic variety of order $\mathrm{dim}(V)$. By a {\em generic point} $v$ of $(V,s)^{\partial}$ we mean a generic point of $v\in V(\mathcal{U})$ in the algebro-geometric sense,i.e., $\mathrm{dim}(V)=tr.deg(K(v)/K)$, such that $\partial(v)=s(v)$. %Some important properties of $(V,s)^{\partial}$, which in a way explains it usefulness, are as follows (cf.
From \cite{Buium} we have that $(V,s)^{\partial}$ is Zariski dense in $V(\mathcal{U})$.   

As a remark for who knows model theory, from \cite{Pillay3}: if $K$ is a differentially closed field (or even an algebraically closed differential field), any definable subset of $Y\subseteq K^n$ of finite Morley rank and Morley degree 1, is, up to definable bijection, of the form $(V,s)^{\partial}$ for some irreducible algebraic D-variety $(V, s)$.

\begin{rmk}\label{linear}
 Any algebraic differential equation can birationally be expressed as a $D$-variety. An easy example is the following: if the differential equation is of the form $y^{(n)}=f(t,y,y',\dots,y^{(n-1)})$, with $f$ polynomial (or rational) with coefficient from $\mathbb{C}$, we consider the change of variables $y_i=y^{i}$, for $i\leq n$. Then $\mathbb{A}_{\mathbb{C}(t)}^n$ together with the $s:(y_0,y_1,\dots,y_n)\mapsto(y_1,y_2,\dots,f(t,y_0,\dots,y_{n-1}))$ is our $D$-variety. 
In general an algebraic differential equation is of the form $F(t,y,y',\dots,y^{(n)})=0$, with $F$ algebraic; it is then sufficient to differentiate the equation a finite number of times and, after the change of variables $y_{i}=y^{(i)}$, we obtain a system of equations of order $1$, and it is easy now to define the corresponding $D$-variety.

\end{rmk}

\begin{exmp}\label{pp1}
Let $K=\mathbb{C}(t)$. The first Painlev\'e equation is given by $y''=6y^2+t$. It is naturally defined over $\mathbb{C}(t)$ and if we set $y_1:=y$ and $y_2:= y'$ and let $s(y_1,y_2)=(y_2,6y_1^2+t)$, we can rewrite the equation as $\partial(y_1,y_2)=(y_2,6y_1^2+t)=s(y_1,y_2)$. Hence for this equation we consider the $D$-variety $(\mathbb{A}_{\mathbb{C}(t)}^2,s)$ (over $\mathbb{C}(t)$) and  the solution set is thus: 
\[(\mathbb{A}_{\mathbb{C}(t)}^2,s)^{\partial}=\{(y_1,y_2)\in\mathbb{A}_{\mathcal{U}}^2:\partial(y_1,y_2)=s(y_1,y_2)\}.\]
\end{exmp}

A less trivial example of a $D$-variety is the following:
\begin{exmp}\label{exell}
To write the elliptic equation $(y')^2=4y^3+g_2y+g_3$, with $g_2,g_3\in\mathbb{C}$, as a $D$-variety, we firstly derive both sides obtaining the equation $y''=6y^2+\frac{g_2}{2}$, then we name $y_0:=y$, $y_1:=y'$ and $y_2:=y''$. Thus we get the $D$-variety $(V,s)$, with $V\ssetq \mathbb{A}^3$ the algebraic variety defined by $y_1^2=4y_0^3+g_2y_0+g_3$ together with  $s(y_0,y_1,y_2)=\left(y_1,y_2,6y_0^2+\frac{g_2}{2}\right)$  
\end{exmp}

We remind that we work birationally and so we need the following definition:  
$(V,s)^{\partial}$ is \emph{generically internal} (resp. \emph{generically almost internal}) to $\mathcal{C}$ if $(V\setminus V',s)^{\partial}$ is internal (resp. almost internal) to $\mathcal{C}$, where $(V',s\restr V')$ is some proper $D$-subvariety of $(V,s)$.

The following is well-known (cf. \cite{BigPillay}):
\begin{fct}\label{fact1}Suppose that $K$ is a differential field and that $(V,s)$ is an irreducible algebraic $D$-variety of dimension $n$ (over $K$). Then:
\begin{enumerate}
\item $(V,s)^{\partial}$ is generically internal to $\mathcal{C}$ (with $m=n$ in the definition) over some $L>K$ if and only if for a generic point $y$ of $(V,s)^{\partial}$ over $L$, $y$ is contained in the differential field generated by $L$ and $\mathcal{C}$, i.e. $y\in dcl(L,C)$.
\item $(V,s)^{\partial}$ is generically almost internal to $\mathcal{C}$ (with $m=n$) over some $L>K$ if and only if for a generic point $y$ of $(V,s)^{\partial}$ over $L$, $y$ is contained in the algebraic closure of the differential field generated by $L$ and $\mathcal{C}$, i.e. $y\in acl(L,C)$.
\end{enumerate}
\end{fct}

\begin{rmk}\label{rem1}Let $(V,s)$ be an irreducible $D$-variety ($dim(V)=n$) defined over $K=\mathbb{C}(t)$ and suppose that $(V,s)^{\partial}$ is generically almost internal to $\mathcal{C}$ over some $L>K$, then: 
\begin{enumerate}
\item One can witness almost internality over $L=K(\bar{a})$ where $\bar{a}$ is a tuple of elements of $(V,s)^{\partial}(K^{diff})$ and $K^{diff}$ denote the differential closure of $K$.
\item There are $L$-definable functions $f_1,\ldots,f_n:(V,s)^{\partial} \rightarrow \mathcal{C}$ such that for any generic point $y\in(V,s)^{\partial}$, $f_1(y),\ldots,f_n(y)$ are algebraically independent over $L$.
\end{enumerate}
\end{rmk}
\begin{proof}We give a sketch of the proof for the model theorist:\\
(1)  We proceed as in \cite{Poizat}. We may assume that $L$ is a model, that is $L$ is differentially closed. So for any generic $y\in(V,s)^{\partial}$  we have that
\[y\in acl(\overline{c},L)\text{  and  } y\indep[K]L, \]
where $\overline{c}$ is a tuple from $\mathcal{C}$. Using the definability of types we can replace $L$ by $K^{diff}$, that is we have that $y\in acl(K^{diff},\overline{c})$. \\
Let $N$ be the differential field generated by $K$, $\mathcal{C}\cap K^{diff}$ and elements of $(V,s)^{\partial}(K^{diff})$. Then it not hard to see that $tp(y\overline{c}/K^{diff})$ is determined by its restriction to $N$  (c.f. \cite{Poizat} p. 261), so $tp(y/N\cup\{\bar{c}\})$ has a unique extension over $K^{diff}\cup\{\overline{c}\}$. Hence as $tp(y/K^{diff}\cup\{\overline{c}\})$ is algebraic, $y\in acl(N,\overline{c})$, i.e for some tuples $\overline{c}'$ in $\mathcal{C}\cap K^{diff}$ and $\overline{a}$ in $(V,s)^{\partial}(K^{diff})$, $y\in acl(K,\overline{c},\overline{c}',\overline{a})$.\\ 
(2) From the proof of part (1) we have that for generic $y\in(V,s)^{\partial}$ over $L=K(\overline{a})$, $y\in acl(L,\overline{c})$, where $\overline{c}\in\mathcal{C}$ and $\overline{a}$ is some tuple in $(V,s)^{\partial}(K^{diff})$. Using the Tarski-Vaught theorem, we can replace $\overline{c}$ by constants $\overline{d}$ in $L(y)^{diff}$ and since $\mathcal{C}\cap L(y)^{diff}$ is contained in $L(y)^{alg}$, we have that $\overline{d}\in L(y)^{alg}$.\\
Let $\{\overline{d}_1=\overline{d},\ldots,\overline{d}_r\}$ be the set of conjugates of $\overline{d}$ over $L(y)$ and let $(d_1,\ldots,d_s)$ be the code of $\{\overline{d}_1,\ldots,\overline{d}_r\}$ which exist in $L(y)^{diff}$ by elimination of imaginaries. Then we have that $y\in acl(d_1,\ldots,d_s,L)$ and each $d_i\in dcl(y,L)$. So $y$ and $(d_1,\ldots,d_s)$ are interalgebraic over $L$. Finally, as $tr.deg(\mathbb{Q}(L,y)/\mathbb{Q}(L))=n$ we can choose, $d_1,\ldots,d_n$ (maximal) algebraic independent over $L$ so that $d_i\in dcl(y,L)$ and $y\in acl(d_1,\ldots,d_n,L)$.\\ %ssume that $(V,s)^{\partial}$ is {\color{red}generically almost internal to $\mathcal{C}$ over $L=\hat{K}$}. As in Fact \ref{fact1}, we have
\end{proof}

\section{First integrals and Integrability}

We now describe a traditional notion of integrability: having enough first integrals. The first integrals are functions from the phase space of a differential equation to the constants, analytic on some nonempty open set; where, for an autonomous differential equation in $y(t)$ of order $n$, the phase space is the $n+1$ dimensional space with coordinates $(t,y,y',\dots,y^{(n)})$. With a change of coordinate as in Remark \ref{linear} we obtain a system of equations of order $1$, the phase space has then coordinates $(t,y_0,\dots,y_{n-1})$.
 
In our context from a $D$-variety we will recover a fibred space (that in the case of an autonomous equation will correspond to the usual phase space) $\mathcal{V}\rightarrow \mathbb{A}^1$ together with a section $s_{\mathcal{V}}$, where it is possible to define a first integral. This duality of a differential equation as a $D$-variety and as a fibred space justifies the geometric approach used in this article.\\

Let $(V, s_V)$ be an irreducible $D$-variety defined over $\mathbb{C}(t)$ ($t\in\mathcal{U}$ such that $\partial t=1$) of dimension $n$. 
If locally, we consider the polynomials generating $I(V)\subseteq\mathbb{C}(t)[\overline{x}]$, the ideal of $V$, as polynomials in $\mathbb{C}$ (that is we allow $t$ to vary), we obtain a complex algebraic variety $\mV$ of dimension $n+1$. The condition $\partial(t)=1$ induces a complex algebraic variety $S\ssetq \mathbb{A}_{\mathbb{C}}^1$ together with the constant section $1_S:S\rightarrow TS$, $t\mapsto (t,1)$; naturally we get a projection $\mathcal{V}\rightarrow S$. The section $s_V$ induces a section $s_{\mV}:\mV\rightarrow T\mV$ that agrees with $1_S$ and is therefore defined by $s_{\mV}=(1,s_V)$. We can thus represent an algebraic differential equation as a fibration  $p:(\mathcal{V},s_{\mathcal{V}})\rightarrow (S,1)$ of which $(V, s_V)$ is the generic fibre of $p$\\
%It turns out that $V$ can be viewed as the generic fibre of some fibration $p:(\mathcal{V},s_{\mathcal{V}})\rightarrow (S,1)\subseteq\mathbb{A}^1_{\mathbb{C}}$, where $\mathcal{V}$ is an irreducible complex algebraic variety of dimension $n+1$, $S$ is open in $\mathbb{A}^1_{\mathbb{C}}$ and where $s_{\mV}$ is a regular vector field on $\mathcal{V}$ obtained from $s_V$ and by lifting the vector field $1$ on $S$, i.e $s_{\mV}=(1,s_V)$. 
\indent The equality $\partial(\bar{v})=s_V(\bar{v})$ translates to the differential equation $\frac{d\bar{y}}{dt}=s_{\mV}(t,\bar{y})$ on local analytic sections $\bar{y}$ of $p$. So a solution in our context is a local analytic section $\bar{g}:D\rightarrow\mV$ of $p$ on some domain $D$ in $S$ and such that $d\bar{g}(\frac{d}{dt})=s_{\mV}$. Furthermore, $\bar{g}$ defines an algebraic integral curve $C_{\bar{g}}$ of $s_{\mV}$, which is just the Zariski closure in $\mV$ of the graph of $\bar{y}=\bar{g}(t)$.
Of course $(\mathcal{V},s)$ can also be seen as a $D$-variety  defined over $\mathbb{C}$; however we will use this fact only in the next section. \\

In the autonomous case this construction can be simplified: for $(V,s_V)$ an irreducible algebraic (or rational) $D$-variety defined over $\mathbb{C}$, we look at the trivial fibration $p:\mathcal{V}=\mathbb{A}^1_{\mathbb{C}}\times V\rightarrow\mathbb{A}^1_{\mathbb{C}}$, so that $\mathcal{V}$ is again an irreducible complex algebraic variety and in this case we have $(\mathcal{V},s_{\mV})=(\mathbb{A}^1\times V,1\times s_V)$, the phase space.\\

Let then $(\mathcal{V},s_{\mV})$ be a complex algebraic variety of dimension $n+1$ with a distinguished section $s_{\mathcal{V}}$ to its tangent bundle; a \textit{rational first integral} of $(\mathcal{V},s_{\mV})$ on some Zariski open $U\subset\mV$ is a rational function $h:\mV\rightarrow \mathbb{C}$ such that $dh(s_{\mV}(v))=0$ for all $v\in U$. For $c\in Im(h)$, we will denote by $H_c$ the closed subvariety of $\mathcal{V}$ of dimension $n$ defined by $h^{-1}(c)$. Given $(\mV,s_{\mV})\rightarrow (S,1)$, let $\bar{g}:D\rightarrow\mV$ be a local solution. Then, if $h:\mV\rightarrow \mathbb{C}$ is a rational first integral, we see that $h\circ\bar{g}:D\rightarrow\mathbb{C}$ is constant as $0=dh(s_{\mV})=\frac{d(h\circ\bar{g})}{dt}$. 
 
%\begin{rmk}
% An important remark is that, given a rational first integral $h$, a solution (restricted to $U$) $y:U\rightarrow \mathcal{V}$ of $(\mathcal{V},s)$ (or, more precisely its integral curve $Im(y)$) is contained in $H_c$ for some constant $c$.
%\end{rmk}
%\begin{proof}
% Given $t_0$, $t_1$, a solution $y$ determines points $y(t_0)$, $y(t_1)$. We then want to show that $h(y(t_0))=h(y(t_1))$ for each $t_0,t_1\in\mathbb{A}^1$. 

%Suppose not: i.e., $h(y(t_0))\neq h(y(t_1))$ for some $t_0,t_1$. Then the function $h$ along the $Im(y)$ is not constant, i.e., $d/d t h(y(t))\neq 0$; but $d/dt h(y)=dh\circ dy(d/dt)=dh(s)=0$ since $y$ is a solution and $h$ is a first integral. Contradiction.
%\end{proof}

Suppose that $h_1,\dots,h_n$, $h_i:U_i\rightarrow \mathbb{C}$, are $n$ rational first integrals of $(\mathcal{V},s_{\mV})$. 
We say that $h_1,\dots,h_n$ are \textit{independent} if there is no (nontrivial) function $\theta\in \mathbb{C}[x_1,\ldots,x_n]$ such that for some generic point $v\in \mV$, $\theta(h_1(v),\dots,h_n(v))=0$.

\begin{lem}\label{indepen}
Let $(\mathcal{V},s_{\mV})$ be as above and suppose that $h_1,\dots,h_n$ are $n$ independent rational first integrals of $(\mV,s_{\mV})$. Suppose $y\in\mV$ is generic and let $c_1,\ldots,c_n\in\mathbb{C}$ be such that $h_i(y)=c_i$ for all $i$. Then
\begin{enumerate}
\item For $i\neq j$, $\dim(H_{i,c_i}\cap H_{j,c_j})< n$, where the $H_{i,c_i}$ are the subvarieties defined by $h_i^{-1}(c_i)$ for all $i$.
\item $\dim(\cap_{i<n}H_{i,c_i})=1$.% and $\bigcap_{i\leq n}TF_{i,c_i}=F=<s>$. 
%Moreover there is no open set in $\cap_{i<n}F_{i,c_i}$ such that the projection on the first coordinate is a single point $t_0$.
\end{enumerate}
\end{lem}
\begin{proof}(1) Let $\overline{d}$ be the finite tuple of complex parameters such that $K=\mathbb{Q}(\overline{d})$ is the smallest field of definition of $(\mV,s_{\mV})$. We assume for simplicity that all the $h_i$'s are also defined over $K$. As $y$ is a generic point of $(\mV,s_{\mV})$, $tr.deg(K(y)/K)=n+1$ and as $y\in H_{i,c_i}$, $tr.deg(K(y)/K(c_i))=n$ for all $i$ ($c_i\not\in K^{alg}$).\\
Suppose for contradiction that for $i\neq j$, $\dim(H_{i,c_i}\cap H_{j,c_j})=n$, that is, suppose that $tr.deg(K(y)/K(c_i,c_j))=n$. Then, without loss of generality, we have that $c_j\in K(c_i)^{alg}$. Hence there is $\theta(c_i,x)$ in $K(c_i)[x]$ such that $\theta(c_i,c_j)=0$, i.e., $\theta(h_i(y),h_j(y))=0$ and this contradicts independence of the first integrals.\\
(2) We only need to observe that $\dim(\cap_{i<n}H_{i,c_i})\geq 1$ as $\cap_{i<n}H_{i,c_i}$ always intersect the integral curve through $y$ in infinitely many points. Using induction together with a similar argument as in the proof of part (1) one has that $\dim(\cap_{i<n}H_{i,c_i})\geq 1$ and so we are done.

\end{proof}

Integrability is then having enough independent first integrals:

\begin{defn}
 We say that $(\mV,s_{\mV})\rightarrow (S,1)$  is \tit{algebraically integrable} if it has $n$ independent rational first integrals.
\end{defn}
 
In \cite{Chat}, algebraic integrability was introduced in a slightly different way, it is however not hard to see that the following is true.
\begin{fct}
$(\mV,s_{\mV})\rightarrow (S,1)$ is algebraically integrable if and only if through every point in a Zariski open subset of $\mV$ there passes an algebraic integral curve of $s_{\mV}$.
\end{fct}
The statement on the right hand side is precisely the definition given in \cite{Chat}.
%{\color{red}It should be true! Isn't it? Do we prove it?}

%{\color{red} am I wrong or we didn`t even say what a solution is in this context? we probably don`t use it, although I think we`ll do in chapter 5... we`ll think about it later on...}

\section{Internality as an integrability notion}

In \cite{Pillay}, Pillay claims the following result,
\begin{claim}
Let $(V,s)$ be a rational $D$-variety defined over $\mathbb{C}(t)^{alg}$. Let $p:(\mathcal{V},s_{\mV})\rightarrow (S,1)$ be the corresponding fibration (as in Section 3) of complex varieties, $S\subseteq\mathbb{A}^1_{\mathbb{C}}$. Then $(V,s)^{\partial}$ is generically internal to $\mathcal{C}$ if and only if there are complex varieties with vector fields $(S',s_0)$ and $(F,0)$, and a dominant map $(S',s_0)\rightarrow(S,1)$ over $\mathbb{C}$, such that $(\mathcal{V},s_{\mV})\times_{(S,1)}(S',s_0)$ birationally embeds (over $\mathbb{C}$) into $(F,0)\times(S',s_0)$.
\end{claim}
Where $(F,0)$ denotes the complex variety $F$ equipped with the $0$-section of the tangent bundle. But a proof is not given. 

We shall obtain a result generalizing this, although our focus is more on the meaning of internality in the sense of existence of first integrals, and so we provide an explicit description of such first integrals.

In this direction the only result is by Chatzidakis-Hrushovski in \cite{Chat}, and it is limited to the autonomous case:
\begin{prop}
Let $(V,s)$ be a rational $D$-variety defined over $\mathbb{C}$. Let $p:(\mathcal{V},s_{\mV})\rightarrow (S,1)$ be the corresponding fibration, $S\subseteq\mathbb{A}^1_{\mathbb{C}}$. Then $(V,s)^{\partial}$ is generically almost internal to $\mathcal{C}$ over $\mathbb{C}(t)^{alg}$ if and only if $(\mathcal{V},s_{\mV})$ is algebraically integrable.
\end{prop}

We generalize this result to the non-autonomous case (i.e. for irreducible $D$-varieties defined over $\mathbb{C}(t)$). Moreover almost internality over $\mathbb{C}(t)^{alg}$ is quite restrictive; we want to allow the possibility of witnessing almost internality using any parameters from the universal domain $\mathcal{U}$. From Remark \ref{rem1}, this will occur in $L=\mathbb{C}(t,\bar{a})$ for some tuple $\bar{a}$ of $ (V,s)^{\partial}(\mathbb{C}(t)^{diff})$. Let $W$ be the (algebraic) locus of $\bar{a}$ over $\mathbb{C}(t)^{alg}$ and let $s_W$ be the rational function such that $\partial(\bar{a})=s_W(\bar{a})$. $(W,s_{W})$ is then a $D$-subvariety of $(V,s)^m$ for some $m\in\mathbb{N}$, defined over $\mathbb{C}(t)^{alg}$.

%We now aim at generalizing this result. We will look at the non-autonomous case, i.e. we assume that $(V,s)$ is a rational $D$-variety defined over $\mathbb{C}(t)^{alg}$. Intuitively almost internality over $\mathbb{C}(t)^{alg}$ will correspond to some integrability condition on $(\mathcal{V},s_{\mV})$. However this is quite restrictive; we need to allow witnessing almost internality over some arbitrary extension $L$ of $\mathbb{C}(t)^{alg}$. It turns out that in that case working with $(\mathcal{V},s_{\mV})$ is not enough as we now discuss.
%\par So let $(V,s_V)$ be a $D$-variety defined over $\mathbb{C}(t)^{alg}$ and as before consider the fibration $p_1:(\mV,s_{\mV})\rightarrow(S_1,1)$ of complex algebraic varieties of which $(V,s_V)$ is the generic fibre. Suppose we need to witness internality over some extension $L\supseteq\mathbb{C}(t)^{alg}$. From Remark \ref{rem1}, $L=\mathbb{C}(t)^{alg}\gen{\bar{a}}=\mathbb{C}(t)^{alg}(\bar{a},\ldots,\partial^r(\bar{a}))$ for some $\bar{a}\in\mathcal{U}$ and $r\in\mathbb{N}$. After replacing $\bar{a}$ by $(\bar{a},\ldots,\partial^r(\bar{a}))$ we have that $\partial(\bar{a})\in\mathbb{C}(t)^{alg}(\bar{a})$ and we can assume that $\partial(\bar{a})= \chi(\bar{a})$ for some rational function $\chi(\bar{x})$. $\bar{a}$ is the generic point of an irreducible algebraic variety $W$ defined over $\mathbb{C}(t)^{alg}$ and we have a $D$-variety $(W,\chi)$ defined over $\mathbb{C}(t)^{alg}$. For ease of notation we denote  $(W,\chi)$ with $(W,s_{W})$.
As before we let $(W,s_{W})$ be the generic fibre of some fibration $p_2:(\mW,s_{\mW})\rightarrow(S_2,1)$ of irreducible complex varieties. Finally, the ``correct" variety to work with is the fibered product over $S=S_1\cap S_2\subseteq\mathbb{A}^1_\mathbb{C}$, $(\WSV,\pi_1,\pi_2)$, with the %usual
commutative diagram:
\begin{equation*}
\xymatrix{&\WSV \ar[ld]^{\pi_2}\ar[rd]_{\pi_1}\ar[dd]^{\pi} &\\ \mW \ar[dr]^{p_2}&&
  \mV \ar[dl]_{p_1}\\& S&}
\end{equation*}

We hence obtain a complex algebraic variety (with a vector field) $(\WSV,s_{\mW\times\mV})$, where $s_{\mW\times\mV}=(1,s_W,s_V)$.

Our previous definition of independence will only be informative about $(\WSV,s_{\mW\times\mV})$, but not about $(\mV,s_{\mV})$, the variety we are interested in. We need thus to strengthen the definition of independence of first integrals, to impose that they are independent also when we restrict our attention to $(\mV,s_{\mV})$:

\begin{defn}
Let $h_1,\dots,h_n$ be rational first integrals of $(\WSV,s_{\mW\times\mV})$. We say that $h_1,\dots,h_n$ are {\em $\mW$-independent} if there is no function\\ $\theta\in\mathbb{C}(t,w)[x_1,\ldots,x_n]$ such that for some generic point $(t,w,v)$ of $\WSV$, $\theta(h_1(t,w,v),\dots,h_n(t,w,v))=0$.  
\end{defn}

An analogue of Lemma \ref{indepen} in this case is as follows:
\begin{lem}\label{indstr}
Let $(\WSV,s_{\mVW})$ be as above and suppose that $h_1,\dots,h_n$ are $n$ $\mW$-independent rational first integrals of $(\WSV,s_{\mVW})$. Suppose $(t,w,v)\in\mVW$ is generic and let $c_1,\ldots,c_n\in\mathbb{C}$ be such that $h_i(t,w,v)=c_i$ for all $i$. Then
\begin{enumerate}
\item For $i\neq j$, $\mathrm{dim}((H_{i,c_i}\cap H_{j,c_j})_{{\restriction}{\mV}})< n$, where the $H_{i,c_i}$ are the subvarieties defined by $h_i^{-1}(c_i)$.
\item $\dim((\cap_{i<n}H_{i,c_i})_{{\restriction}{\mV}})=1$.
\end{enumerate}
\end{lem}
\begin{proof} The proof is similar to that of Lemma \ref{indepen} and we leave it to the reader.
%(1) Let $\overline{d}$ be the finite tuple of complex parameters appearing in the definition of $(\mVW,s_{\mVW})$ and let $K=\mathbb{Q}(\overline{d})^{alg}$. As $(t,w,v)$ is a generic point of $(\mVW,s_{\mVW})$, $tr.deg(K(t,w,v)/K)=n+k+1$ and as $(t,w,v)\in H_{i,c_i}$, $tr.deg(K(t,w,v,c_i)/K)=n+k$ for all $i$. As in Lemma \ref{indepen}, $tr.deg(K(t,w,v,c_i,c_j)/K)<n+k$.\\
%Suppose for contradiction that $\dim(H_{i,c_i}\cap H_{j,c_j})_{{\restriction}{\mV}})=n$, that is suppose that $tr.deg(K(t,v,c_i,c_j)/K)=n$. But then $tr.deg(K(w,c_i,c_j)/K)\leq k-2$. So without loss of generality we have that $c_j\in K(w,c_i)^{alg}$. Hence there is $\theta(w,c_i,x)$ in $K(w,c_i)[x]$ such that $\theta(w,c_i,c_j)=0$, that is \\$\theta(w,h_i(t,w,v),h_j(t,w,v))=0$ and this contradicts $\mW$-independence of the first integrals.\\
%(2) Follows as in Lemma \ref{indepen}(2).
\end{proof}

%\begin{proof} This is just for  
%$(1)$: Let $\dim(\mV)=n+1$ and $\dim(\mW)=k+1$, so $\dim(\WAV)=n+k+1$. Note that $\mathrm{dim}((H_{i,c_i})_{{\restriction}{\mV}})=n$ and $\mathrm{dim}((H_{i,c_i})_{{\restriction}{\mV}})=k$.
%\par By Lemma \ref{indepen} we have that for any $i\neq j$ and any $(c_1,c_2)\in Im(h_i)\times Im(h_j)$, $\dim((H_{i,c_1}\cap H_{j,c_2}))< n+k$  in $\WSV$. For contradiction, suppose for some $(c_1,c_2)\in Im(h_i)\times Im(h_j)$, $\dim((H_{i,c_1}\cap H_{j,c_2})_{{\restriction}{\mV}})= n$, that is without loss of generality $(H_{i,c_1}\cap H_{j,c_2})_{{\restriction}{\mV}}=(H_{i,c_1})_{{\restriction}{\mV}}$. Then we would have $\dim((H_{i,c_1}\cap H_{j,c_2})_{{\restriction}{\mW}})<k$ (remember that $\dim(\mV)+\dim(\mW)=\dim(\WSV)+1$). Define $\theta(x_1,x_2)=c_2x_1-c_1x_2$; observe though that $\theta(h_i(t,w,v),h_j(t,w,v))$ is not necessarily $=0$ for $(t,w,v)\in H_{i,c_1}$.

%Since $\dim(\pi_2(F_{i,c_1}\cap F_{j,c_2}))<k$,  $\pi_2(F_{i,c_1}\cap F_{j,c_2})$ determines a rational subvariety of $\mW$. 
%Let $g(t,w)$ be the rational function $g:\mW\rightarrow \mathbb{C}$ such that $g^{-1}(c_1)=\pi_2(F_{i,c_1}\cap F_{j,c_2})$.  
 
%Then the function $\theta'(f_i,f_j,t,w)=\theta(f_i,f_j)-g(t,w)=0$ for all $(t,w,v)\in Z$ large subset of $\WAV$, contradicting $\mW$-independence of the $f_i,f_j$'s.

%$(2)$: follows from $(1)$ with essentially the same proof as Lemma \ref{indepen}.
%\end{proof}
\vskip 0.4 cm
The intuitive idea of the lemma above is that the set of preimages of a first integral $f_i$: $\{F_{i,c}\}_{c\in\mathbb{C}}$ will produce a ``slicing'' of $\mathcal{W}\times_S \mathcal{V}$. Having $\mW$-indipendent first integrals means that different first integrals produce slicings that are not parallel on the ``$\mV$-part'' of $\WSV$. Thus, projecting on $\mV$, enough $\mW$-independent first integrals will single out a integral curve of $\mathcal{V}\rightarrow S$.
\vskip 0.4 cm

\begin{lem}\label{funct}
Suppose $(V,s_V)$ is an irreducible $D$-variety defined over $\mathbb{C}(t)$. A rational first integral of $(\WSV,s_{\mW\times\mV})$ corresponds to a definable function (in $DCF_0$) from $(V,s_V)^{\partial}$ to the constants $\mathcal{C}$ defined over $\mathbb{C}(t,\bar{a})$, where $(t,\bar{a})$ is any generic point of $(\mW,s_{\mW})^{\partial}$.
\end{lem}
\begin{proof}
First, suppose $h:\WSV\rightarrow\mathbb{C}$ is a rational first integral of $\WSV$, i.e. for some open $U\subseteq\WSV$, $dh(s_{\mW\times\mV}(u))=0$ for all $u\in U$. We can extend $h$ to a function $h:(\WSV)(\mathcal{U})\rightarrow\mathcal{U}$ and we show that for any $u\in U\cap(\WSV,s_{\mW\times\mV})^{\partial}$, $\partial(h(u))=0$. Of course such a $u$ is of the form $(t,w,v)$ for $t\in S$ and we have that
\[\begin{array}{rl} \partial(h(t,w,v))&=dh(\partial(t,w,v))\\
&=dh(s_{\mW\times\mV}(t,w,v))\\
&=0.
  \end{array}\]
So, in particular, if we choose $(t,\bar{a})$ a generic point of $(\mW,s_{\mW})^{\partial}$,as we work birationally, we have that $f=h(t,\ovl{a},v)$ as a function of $v\in(V,s_V)^{\partial}$ takes $(V,s_V)^{\partial}$ to the constants and is defined over $\mathbb{C}(t,\bar{a})$.
\par Conversely, suppose $f$ is a definable function from $(V,s_V)^{\partial}$ to the constants defined over $\mathbb{C}(t,\bar{a})$ where $(t,\bar{a})$ is a generic point of $(\mW,s_{\mW})^{\partial}$. So $f$ is a function of the form $f(t,\bar{a},\partial(\bar{a}),\ldots,v,\partial(v),\ldots)$ and if we replace the occurrence of $\partial(\bar{a})$ by $s_W(\bar{a})$ and that of $\partial(v)$ by $s_V(v)$, $f$ is a function $h$, rational in $t$, $\bar{a}$ and $v$. So in particular, as $(t,\bar{a})$ is generic, $h$ is a well define function on some $U\subseteq\WSV$, where $U\subseteq\pi^{-1}_1(p_1^{-1}(V))$, and it remains to show that it is a first integral. But for any $u=(t,w,v)\in U\cap(\WSV,s_{\mW\times\mV})^{\partial}$ we have that
\begin{equation*}
\begin{split}
dh(s_{\mW\times\mV}(t,w,v) )% = df(s_{\mW}(t,w),s_{\mV}(t,v))\\
% & = df(\partial(t,w),\partial(t,v))\\
 & = dh(\partial(t,w,v)\\
 & = \partial(h(t,w,v))\\
& = 0.
\end{split}
\end{equation*}
As $(\WSV,s_{\mW\times\mV})^{\partial}$ is Zariski dense in $(\WSV,s_{\mW\times\mV})$, we are done.
\end{proof}

\begin{lem}\label{AlgInd}
Let $h_1,\ldots,h_n:\WSV\rightarrow\mathbb{C}$ be first integrals of $\mW\times_{\mathbb{A}^1}\mV$. Then $h_1,\ldots,h_n$ are $\mW$-independent  first integrals if and only if
for the corresponding definable functions $f_1,\ldots,f_n:(V,s_V)^{\partial}\rightarrow \mathcal{C}$ given in Lemma \ref{funct} and for any generic point $y\in(V,s_V)^{\partial}$, we have that $f_1(y),\ldots,f_n(y)$ are algebraically independent over $\mathbb{C}(t,\bar{a})$, $(t,\bar{a})$ a generic point of $(\mW,s_{\mW})^{\partial}$.
% and for (almost) all $y\in (V,s_V)$ .
\end{lem}
\begin{proof}
\par Left to right. Suppose $h_1,\ldots,h_n$ are $\mW$-independent and for contradiction suppose 
there exists a polynomial $G\in\mathbb{C}(t,\bar{a})[x_1,\ldots,x_n]$ with $(t,\ovl{a})$ a generic point of $(\mW,s_{\mW})^{\partial}$, such that 
\[G(f_1(t,\ovl{a},y),\ldots,f_n(t,\ovl{a},y))=0,\]
on a generic point $y$ of $(V,s_V)^{\partial}$. If we replace the occurrence of $\partial(\ovl{a})$ by $s_W(\ovl{a})$ and that of $\partial(y)$ by $s_V(y)$ in the $f_i$'s, we get $G(h_1(t,\ovl{a},y),\ldots,h_n(t,\ovl{a},y))=0$ where $G\in\mathbb{C}(t,\ovl{a})[x_1,\ldots,x_n]$. But $(t,\ovl{a},y)$ is then a generic point of $\mW\times_{\mathbb{A}^1}\mV$, with $G(h_1(t,\ovl{a},y),\ldots,h_n(t,\ovl{a},y))=0$, contradicting $\mW$-independence of the $h_i$'s. 
\par Right to left follows from a similar argument. 
\end{proof}

\begin{thm}\label{mainth} Suppose $(V,s_V)$ is an irreducible (rational) $D$-variety of dimension $n$ defined over $\mathbb{C}(t)$. Then $(V,s_V)^{\partial}$ is generically almost internal to $\mathcal{C}$ if and only if there exists a subvariety $\mathcal{W}$ of $\mathcal{V}^m$, for some $m$, together with a vector field $s_{\mW}$ such that $(\WSV,s_{\mW\times\mV})$ has $n$ $\mW$-independent rational first integrals.
\end{thm}
\begin{proof}
Left to right. Suppose that $\mW$ is as given in the statement and let $h_1,\ldots,h_n:\WSV\rightarrow\mathbb{C}$ be $n$ $\mW$-independent first integrals of $\mW\times_{\mathbb{A}^1}\mV$.\\
Let $(t_0,\bar{a},y)\in(\WSV,s_{\mW\times\mV})^{\partial}$ be generic and let $(c_1,\ldots,c_n)\in\mathcal{C}^n$ be such that $f_i(y)=h_i(t_0,\bar{a},y)=c_i$ for each $i$ and $f_i$ denote the definable function given by Lemma \ref{funct}. Note that by Lemma \ref{AlgInd}, the $f_i(y)$'s are algebraically independent over $\mathbb{C}(t_0,\bar{a})$. We proceed as in the proof of Lemma \ref{indepen} and show that $\dim(\cap_i F_{i,c_i})=0$, where $F_{i,c_i}=f_i^{-1}(c_i)$.\\
As $y$ is a generic point of $(V,s_V)^{\partial}$ over $\mathbb{C}(t_0,\bar{a})$, $tr.deg(\mathbb{C}(t_0,\bar{a},y)/\mathbb{C}(t_0,\bar{a}))=n$. For each $i$, by construction, as $y\in F_{i,c_i}$ we have that $tr.deg(\mathbb{C}(t_0,\bar{a},y)$\\$/\mathbb{C}(t_0,\bar{a},c_i))=n-1$ (the $f_i(y)$'s are transcendental over $\mathbb{C}(t_0,\bar{a})$). We claim that for any $i\neq j$, $tr.deg(\mathbb{C}(t_0,\bar{a},y)/\mathbb{C}(t_0,\bar{a},c_i,c_j))=n-2$. Otherwise, without loss of generality, we would have that $c_i\in\mathbb{C}(t_0,\bar{a},c_j)^{alg}$, so that there exist $\theta(x,c_j)$ in $\mathbb{C}(t_0,\bar{a},c_j)[x]$ with $\theta(c_i,c_j)=0$. This would however contradict the fact that $c_i$ (=$f_i(y)$) and $c_j$ (=$f_j(y)$) are algebraically independent over $\mathbb{C}(t_0,\bar{a})$.\\
Finally, by induction, $tr.deg(\mathbb{C}(t_0,\bar{a},y)/\mathbb{C}(t_0,\bar{a},c_1,\ldots,c_n))=n-n=0$. So the set $\cap_i F_{i,c_i}=\{v\in V(\mathcal{U}):\wedge_i f_i(v)=c_i\}$ is finite. Taking as formula $\wedge_i f_i(v)=c_i\wedge \partial(v)=s(v)$ the result follows.
\par Right to left: from Remark \ref{rem1} we have definable functions $\ovl{f}=(f_i)_{i\in n}:(V_t,s)^{\partial} \rightarrow \mathcal{C}^n$ over $L=\mathbb{C}(t,\bar{a})$ which witness generic almost internality of $(V,s_V)^{\partial}$ to $\mathcal{C}$. Since the new parameters $\ovl{a}$ appear, we need to consider the fibration $(\mW,s_{\mW})\rightarrow(\mathbb{A}^1_\mathbb{C},1)$ complex varieties, with generic fibre the $D$-variety $(W,s_W)$ (over $L=\mathbb{C}(t)^{alg}$) for which $\ovl{a}$ is a generic point. We thus work with the fibration  $(\mW\times_{\mathbb{A}^1}\mV,s_{\mW\times\mV})\rightarrow(\mathbb{A}^1_\mathbb{C},1)$ with vector field $s_{\mW\times\mV}(t,w,v)=(1,s_W(t,w),s_V(t,v))$. 
\par We only have to show that the $f_i$'s give rise to independent first integral of $(\mW\times_{\mathbb{A}^1}\mV,s_{\mW\times\mV})$. But by Lemma \ref{funct}, since the $f_i$'s are defined over $\mathbb{C}(t,\bar{a})$ for some generic point $(t,\bar{a})$ of $(\mW,s_{\mW})^{\partial}$ we have that they correspond to $n$ first integrals $h_1,\ldots,h_n$ of $\mW\times_{\mathbb{A}^1}\mV$. All that remain is to show that the $h_i$'s are independent. By Lemma \ref{AlgInd} this is true if we have that for any generic point $y\in(V,s_V)^{\partial}$ the $f_1(y),\ldots,f_n(y)$ are algebraically independent over $\mathbb{C}(t,\ovl{a})$. But this was already proved in Remark \ref{rem1}(2). Thus the $h_i$'s are $n$ independent first integrals of $\WAV$.
\end{proof}

%This fact implies that if $(\mathcal{V},s_1)$ is algebraically integrable, we have definable functions from $(V_t,s)^{\partial}$ to the constants $\mathcal{C}$.

\vskip 0.5 cm

\begin{exmp}
 A basic example to describe the result is the following linear equation: $y''=4y'-4y$ (equivalently, $s(t,y,y')=(1,y',4y'-4y)=(s_1(y),s_2(y'))$). Its solution is $y(t)=c_1e^{2t}+c_2te^{2t}$. It is clearly internal to the constants, in fact if we consider the field $\mathbb{C}(t)(e^{2t})$ we can express the general solution as a rational function with elements of the field plus constants.

We show now how to find the variety $\mW$ and how to express the first integrals. One first integral is $h(y,y')=\frac{y'-2y}{e^{2t}}$. By considering the geometric variety $(\mW,s_{\mW})$ to be $
                 \mW \rightarrow \mathbb{A}^1
                $ with $s_{\mW}(z)=2z$ (corresponding to the differential equation $z'=2z$ whose solution is $e^{2t}$), we can construct $
                 \WAV \rightarrow \mathbb{A}^1
               $,  and, denoted by $a$ an element of $\mW$, check that $h(t,a,y,y')=\frac{y'-2y}{a}$ is a first integral.

It amounts to show that $dh(s_{\mV})=0$ and indeed\\ $ s_{\mW}(a) \partial h/\partial a+s_1(y)\partial h/\partial y + s_2(y') \partial h/\partial y' =\frac{2y-y'}{a^2}2a -\frac{2y'}{a}+\frac{4y'-4y}{a}=0$.
\end{exmp}

\begin{exmp}

Consider the elliptic equation $(y')^2=4y^3+g_2y+g_3$. We saw in Example \ref{exell} that this can be expressed as a $D$-variety $$\left\{\begin{array}{l}
(y_1)^2=4y_0^3+g_2y_0^2+g_3\\
\partial y_0=y_1\\
\partial y_1=y_2\\
\partial y_2=6y_0+\frac{g_2}{2}.\\
                                                                                                                                          
                                                                                                                                         \end{array}\right.
$$

From \cite{PillSok} Example 0.8 we know that after fixing a generic solution $(\alpha,\alpha')\in (V,s)^{\partial}$, any solution $(\beta,\beta')\in  (V,s)^{\partial}$ is obtained as $(\beta,\beta')=(\alpha,\alpha')\oplus (c,d)$, where $(c,d)\in V(\mathbb{C})$ and $\oplus$ is the usual elliptic curve operation. 

Thus we have a map $(V,s)^{\partial}\rightarrow E(\mathbb{C})\sset \mathbb{C}^2$ witnessing internality; and a first integral of $(y')^2=4y^3+g_2y+g_3$ given by $h:(y,y')\ominus (a,a')$ for $(a,a')$ a particular solution of the same differential equation. 

\end{exmp}

\section{Conclusions and questions}

From Theorem \ref{mainth} and quantifier elimination for $DCF_0$ it is clear that we should be able to obtain a description of the first integrals as rational functions. We show that indeed this is the case and we give a precise reformulation of the theorem in this language. This allows us to connect our result with the Painlev\'e property.
 \vskip 0.2 cm

We need to define in our setting some well known notions: given a fibration $p:(\mathcal{V},s_{\mathcal{V}})\rightarrow (S,1)$, corresponding to a $D$-variety $(V,S_V)$, 
%or, as we saw before, a system of differential equations $E:F(t,\ovl{y},\ovl{y}')=0$, 
let $t_0\in S$ and $\ovl{y}_0\in p^{-1}(t_0)=\mathcal{V}_{t_0}(\mathbb{C})$. A \emph{local analytic solution} of $p:(\mathcal{V},s_{\mathcal{V}})\rightarrow (S,1)$ through $\ovl{y_0}$ is an analytic section $\ovl{y}:H\rightarrow \mV$, where $H$ is some connected domain in $S$ around $t_0$, such that $\ovl{y}(t_0)=\ovl{y}_0$ and $d \ovl{y}/dt=s_{V}(\ovl{y})$. We call $(t_0,\ovl{y}_0)$ the \emph{initial conditions}. 
\vskip 0.2 cm

Given a local analytic solution on $H$ we want to analytically continue it to the entire Riemann sphere $\mathbb{P}^1=\mathbb{C}\cup\{\infty\}$ of which $S$ is a subset.

If there is a point in the domain around which we cannot analytically continue the solution, we call that point a singularity. If it is a pole (also called regular singularity), solutions in the form of an infinite power series can be found applying Frobenius' method, and no multivaluedness arises.

If the singularity is not a pole (nor removable) we call the singularity critical. In presence of a critical singularity an analytic continuation of the solution can produce multivaluedness, but we can opportunely choose a Riemann surface on which the solution can be made single valued. We call the solution obtained after analytic continuation via a suitable Riemann surface a \emph{particular solution} of $p$.

The \emph{general solution} of $p$ is the set $(\mathcal{V},s_{\mathcal{V}})^{sol}$ of all nonsingular particular solutions. Observe that it is not necessarily true that the particular solutions are defined over the same Riemann surface, and if we can find a unique Riemann surface for the general solution of $p$ we say that $p:(\mathcal{V},s_{\mathcal{V}})\rightarrow (S,1)$ has the Painlev\'e property. As it is costumary we improperly talk about singularities of the general solution to mean the singularities of the Riemann sphere.

Observe that $(\mathcal{V},s_{\mathcal{V}})^{sol}$ is the geometric interpretation of $(V,s_V)^{\partial}$, and it differs from $(\mV, s_{\mV})^{\partial}$ where $\mathcal{V}$ is seen as a $D$-variety.
 
\vskip 0.2 cm

For linear differential equations it is possible to determine, from the equation, the points in the Riemann sphere that are singularities for a particular solution, and any other particular solution will have singularities only at these points.  We can thus find a Riemann surface on which all particular solutions can be defined.

For nonlinear differential equations, though, this is trickier: it can happen that different particular solutions have different singularities on the Riemann sphere.

We need thus to distinguish between singularities of the general solution that are fixed, i.e. they are points on the Riemann sphere such that they are singularities for all particular solutions; and those that are movable, i.e. singularities whose position on the Riemann sphere depends on the particular solution chosen, or, equivalently, on the initial conditions.

Since we can deal with movable poles, we can find a Riemann surface on which any solution can be analytically continued. This is what is traditionally called the Painlev\'e property: $p:(\mathcal{V},s_{\mathcal{V}})\rightarrow (S,1)$ has the Painlev\'e property if the general solution has no movable critical singularity.

As Picard has observed, the only real obstruction to being able to define a unique Riemann surface for all the particular solutions of the equation is the existence of movable singularities that cause multivaluedness, i.e. we allow also the existence of non branching critical singularities.

We therefore reformulate the Painlev\'e property as follows:

\begin{defn}
We say that $p:(\mathcal{V},s_{\mathcal{V}})\rightarrow (S,1)$ has the \emph{Painlev\'e property} if the general solution has no movable branch points.
\end{defn}
\vskip 0.3 cm

An interesting example of an equation with the Painlev\'e property is the first Painlev\'e equation, $y''=6y^2+t$ (see Example \ref{pp1}). %The proof of this is ... {\color{red} reference...}

\vskip 0.3 cm

An equation that does not have the Painlev\'e property  is the following: $y'=-\frac{1}{2}y^3$. 
With initial conditions $(t_0,y_0)=(5,\frac{1}{2})$ the solution through $y_0$ is $y(t)=\frac{1}{\sqrt{t-1}}$.

The general solution can be expressed as $y(t)=\frac{1}{\sqrt{t-c}}$ where $c\in\mathbb{C}$, which has a movable branch point and thus does not have the Painlev\'e property.

\vskip 0.3 cm

The existence of $n$ first integrals for an equation of order $n$ implies that the general solution is parametrized by $\mathbb{C}^n$.  This allows to consider an array of constants $\ovl{c}\in \mathbb{C}^n$ as initial conditions to determine particular solutions from the general solution.  We are interested in determining how many particular solutions are singled out by an array of constants.

%{\color{red} issue here about singular solutions: how do we exclued them from the definiton?}

It is thus natural to give the following reformulation of Theorem \ref{mainth}:

\begin{thm}\label{restthm}
 Suppose $(V,s_V)$ is an irreducible (rational) $D$-variety of dimension $n$ defined over $\mathbb{C}(t)^{alg}$. Then $(V,s)^{\partial}$ is generically almost internal to $\mathbb{C}$ if and only if there exists a finite-to-one rational map $(\mathcal{V},s_{\mathcal{V}})^{sol}\rightarrow \mathbb{C}^n$ in $\mathbb{C}(t,y,y',\dots,y^{(n-1)},w_1,\dots,w_k)$, where $w_1,\dots, w_k$ are particular solutions $E$. 
\end{thm}

This map needs not be injective, since generic almost internality guarantees only that for any chosen set of constants $\ovl{c}\in \mathbb{C}$, finitely many solutions will be in the preimage of the rational map.

If we strenghten our assumption to generic internality, we obtain injectivity:

\begin{cor}\label{genint}
 Suppose $(V,s_V)$ is an irreducible (rational) $D$-variety of dimension $n$ defined over $\mathbb{C}(t)^{alg}$. Then $(V,s)^{\partial}$ is generically internal to $\mathbb{C}$ if and only if there exists an injective rational map $(\mathcal{V},s_{\mathcal{V}})^{sol}\rightarrow \mathbb{C}^n$ in\\ $\mathbb{C}(t,y,y',\dots,y^{(n-1)},w_1,\dots,w_k)$, where $w_1,\dots, w_k$ are particular solutions of $E$.
\end{cor}

Thus generic internality implies that, given the general solution, the choice of an array of constants determines a unique particular solution.

\vskip 0.3 cm

\begin{rmk}
 Equations can at this point be brought back from the first order system form to the usual ODE form, $E:F(t,y,y',\dots,y^{(n)})=0$.\\ The conclusion of the statements above is then that our first integrals can be chosen as rational maps in \\$\mathbb{C}(t,y,\dots,y^{(n-1)},w_1(t),\dots,w_s(t),w_1'(t),\dots,w_s'(t),\dots, w_1^{(r)}(t),\dots,w_s^{(r)}(t))$,\\ for some $s,r$ and $w_1(t),\dots,w_s(t)$ particular solutions of $E$.
\end{rmk}

\vskip 0.3 cm
Recalling the examples above, the first Painlev\'e equation $y''=6y^2+t$, is not generically almost internal to the constants. Indeed as pointed out in \cite{NagPil}, if we let $X$ be the solution set in $\mathcal{U}$ of $y''=6y^2+t$, then as $X$ is strongly minimal and of order $2$, it is orthogonal to the $\mathbb{C}$, and this implies non generic almost internality.
\vskip 0.3 cm

The equation  $y'=-\frac{1}{2}y^3$, instead, has first integral $h(t,y)=t-\frac{1}{y^2}$, mapping an element of the general solution $\frac{1}{\sqrt{t-c}}$ to the constant $c$. 

However, given a constant $c_0$, two particular solutions are in the preimage $h^{-1}(c_0)$:  $\frac{1}{\sqrt{t-c_0}}$ and  $-\frac{1}{\sqrt{t-c_0}}$. As $h$ is finite-to-one but not injective, the equation  $y'=-\frac{1}{2}y^3$  is generically  almost internal to the constants but not generically internal.

\vskip 0.3 cm

The notion of generic algebraic integrability is therefore different from the Painlev\'e property, but the notion of generic internality is the intersection of the two:

\begin{cor}
$(V,s_v)^{\partial}$ is generically internal to $\mathbb{C}$ if and only if it is generically almost internal to $\mathbb{C}$ and $p:(\mathcal{V},s_{\mathcal{V}})\rightarrow (S,1)$ has the Painlev\'e property, 
\end{cor}
\begin{proof}
 This is immediate from Corollary \ref{genint}, observing that if there is a branch point, by multivaluedness, there are at least two distinct particular solutions such that any first integrals will send them to the same array of constants in $\mathbb{C}^n$, and this contradicts injectivity of the first integrals.  
\end{proof}

We have thus provided an equivalence between a notion of integrability in the model theoretic and in the traditional sense. In this last section we provided some connection with the Painlev\'e property; a question naturally arising, although a bit vague, is the following:

\begin{question}
 Is there a model theoretic notion equivalent to the Painlev\'e property?
\end{question}

% -----------------------------------------------------------

\end{document}